\documentclass[11pt]{article}

\usepackage{amssymb}
\usepackage[english]{babel}
\usepackage{graphicx,url}
\usepackage{hyperref}

\newtheorem{theo}{Theorem}[section]
\newtheorem{propo}[theo]{Proposition}

\newtheorem{coro}[theo]{Corollary}

\input amssym.def
\newsymbol\rtimes 226F
\newfont{\nset}{msbm10}
\newcommand{\ns}[1]{\mbox{\nset #1}}

\def\A{{\mbox {\boldmath $A$}}}
\def\B{{\mbox {\boldmath $B$}}}
\def\C{{\mbox {\boldmath $C$}}}

\def\D{{\mbox {\boldmath $D$}}}
\def\E{{\mbox {\boldmath $E$}}}
\def\G{\Gamma}

\def\I{{\mbox {\boldmath $I$}}}
\def\J{{\mbox {\boldmath $J$}}}

\def\M{{\mbox {\boldmath $M$}}}
\def\O{{\mbox {\boldmath $O$}}}

\def\U{\mbox{\boldmath $U$}}

\def\R{\ns{R}}

\def\A{{\mbox {\boldmath $A$}}}

\def\matrix0{{\mbox {\boldmath $O$}}}

\def\j{{\mbox{\boldmath $j$}}}

\def\u{{\mbox{\boldmath $u$}}}
\def\v{{\mbox{\boldmath $v$}}}
\def\w{{\mbox{\boldmath $w$}}}

\def\vec0{\mbox{\bf 0}}

\def\dist{\mathop{\rm dist }\nolimits}

\def\exc{\mathop{\rm exc }\nolimits}

\def\tr{\mathop{\rm tr }\nolimits}

\def\sp{\mathop{\rm sp }\nolimits}
\def\span{\mathop{\rm span }\nolimits}

\begin{document}

\title{Algebraic Characterizations of Regularity \\ Properties in Bipartite Graphs
}

\author{A. Abiad, C. Dalf\'{o} and  M.A. Fiol
\\ \\
{\small Dept. de Matem\`atica Aplicada IV} \\
{\small Universitat Polit\`ecnica de Catalunya} \\
{\small Barcelona, Catalonia} \\
{\small (e-mails: {\tt \{aida.abiad,cdalfo,fiol\}@ma4.upc.edu})}\\
\\
\\
Yahya Ould Hamidoune, in memoriam
}


\maketitle



\begin{abstract}
Regular and distance-regular characterizations of general graphs are well-known. In particular, the spectral excess theorem states that a connected graph $\G$ is distance-regular if and only if its spectral excess (a number that can be computed from the spectrum) equals the average excess (the mean of the numbers of vertices at extremal distance from every vertex).
The aim of this paper is to derive new characterizations of regularity and distance-regularity for the more restricted, but interesting, family of bipartite graphs.
In this case, some characterizations of (bi)regular bipartite graphs are given in terms of the mean degrees in every partite set and the Hoffman polynomial. Moreover, it is shown that the conditions for having distance-regularity in such graphs can be relaxed when compared with general graphs. Finally,  a new version of the spectral excess theorem for bipartite graphs is presented.
\end{abstract}

\section{Introduction}
Bipartite graphs are combinatorial objects bearing some interesting symmetries. Thus, their spectra---eigenvalues of their adjacency matrices---are symmetric about zero, as the corresponding eigenvectors come into pairs. Moreover, vertices in the same (respectively, different) independent set are always at even (respectively, odd) distance. Both characteristics have well-known consequences in most properties and parameters of such graphs. Roughly speaking, we could say that the conditions for a given property to be satisfied in a general graph can be somehow relaxed to guarantee the same property for a  bipartite graph.
The goal of this paper is to derive some new results supporting this claim in the framework of regular and distance-regular graphs, for which
several characterizations of combinatorial or algebraic nature are known. Thus, the presented characterizations
of bipartite distance-regular graphs involve such parameters as 
the numbers of walks between vertices (entries of the powers of the adjacency matrix $\A$), the crossed local multiplicities (entries of the idempotents $\E_i$ or eigenprojectors), the predistance polynomials, etc.
For instance, it is known that a regular graph $\G$, with eigenvalues $\theta_0>\theta_1>\cdots > \theta_d$ and diameter $D=d$, is distance-regular if and only if its idempotents $\E_1$ and $\E_d$ belong to the vector space ${\cal D}$ spanned by its distance matrices $\I, \A, \A_2,\ldots \A_d$. In contrast with this,  the same result holds in the case of bipartite graphs, but now only $\E_1\in {\cal D}$ need to be required.
Also, we derive a new version of the spectral excess theorem which gives a quasi-spectral characterization of distance-regularity for (regular) bipartite graphs.

We assume that the reader is familiar with the basic concepts on algebraic graph theory and, in particular, on distance-regular graphs. See, for instance, Cvetkovi\'c, Doob, and Sachs \cite{cds82}, Biggs \cite{biggs}, Brouwer, Cohen, and Neumaier \cite{bcn}, Brouwer and Haemers \cite{bh12},
Van Dam, Koolen, and Tanaka \cite{dkt12}, Fiol \cite{f02}, and Godsil \cite{g93}.

\section{Preliminaries}


Let $\G=(V,A)$ be a (simple and connected) graph on $n=|V|$ vertices, with adjacency matrix $\A$, eigenvalues $\lambda_1\ge \lambda_2\ge \cdots \ge \lambda_n$,  and spectrum
\begin{equation}\label{spG}
\sp \G = \sp \A = \{\theta_0^{m_0},\theta_1^{m_1},\dots,
\theta_d^{m_d}\},
\end{equation}
where the different eigenvalues of $\G$ are in decreasing order,
$\theta_0>\theta_1>\cdots >\theta_d$, and the superscripts
stand for their multiplicities $m_i=m(\theta_i)$. In
particular, note that when $\G$ is $\delta$-regular, the largest eigenvalue is $\theta_0=\lambda_1=\delta$ and has multiplicity $m_0=1$ (as $\G$ is
connected).

Recall also that $\G$ is bipartite if and only if it does not contain odd cycles. Then, its adjacency matrix is of the form
$$
\A=\left(
\begin{array}{cc}
\O & \B \\
\B^\top & \O
\end{array}
 \right).
$$
(Here and henceforth it is assumed that the block matrices have the appropriate dimensions.)
Moreover, for any polynomial $p\in \R_d[x]$ with even and odd parts $p_0$ and $p_1$, we have
\begin{equation}
\label{p(A)}
p(\A) =  p_0(\A)+p_1(\A)
  =
\left(
\begin{array}{cc}
\C & \O \\
\O & \D
\end{array}
 \right) +\left(
\begin{array}{cc}
\O & \M  \\
\M^\top  & \O
\end{array}
 \right).
\end{equation}
Also, the spectrum of $\G$ is symmetric about zero: $\theta_i=-\theta_{d-i}$ and
$m_i=m_{d-i}$, $i=0,1,\ldots,d$. (In fact, a well-known result states that a connected graph $\G$ is bipartite if and only if $\theta_0=-\theta_d$; see, for instance, Cvetkovi\'c, Doob and Sachs \cite{cds82}.) This is due to the fact that, if $(\u|\v)^{\top}$ is a (right) eigenvector with eigenvalue $\theta_i$, then $(\u|-\v)^{\top}$ is an eigenvector for the eigenvalue $-\theta_i$. As shown below, a similar symmetry also applies to the entries of the {\it $($principal\/$)$ idempotents} $\E_i$ representing the projections onto the eigenspaces ${\cal E}_i=\ker (\A-\theta_i\I)$, $i=0,1,\ldots,d$. To see this, first recall that, for any graph with eigenvalue $\theta_i$ having multiplicity $m_i$, its corresponding idempotent can be computed as $\E_i=\U_i\U_i^\top$, where $\U_i$ is the $n\times m_i$ matrix whose columns form an orthonormal basis of ${\cal E}_i$.
For instance, when $\G$ is a $\delta$-regular graph on $n$ vertices
its largest eigenvalue $\theta_0=\delta$ has eigenvector $\j$,
the all-$1$ (column) vector, and corresponding idempotent
$\E_0=\frac{1}{n}\j\j^{\top}=\frac 1n \J$, where $\J$ is the all-$1$ matrix.
Alternatively, we can also compute the idempotents as $\E_i=L_i(\A)$
where $L_i$ is the Lagrange interpolating polynomial of degree $d$
satisfying $L_i(\theta_i)=1$ and $L_i(\theta_j)=0$ for $j\neq i$.
The entries of the idempotents $m_{uv}(\theta_i)=(\E_i)_{uv}$ are called the {\it crossed $uv$-local multiplicities} and, by taking $p=x^{\ell}$, $\ell\ge 0$, they allow us to compute the number of $\ell$-walks between any two vertices:
$$
\label{j-walks}
a_{uv}^{(\ell)}=(\A^\ell)_{uv}=\sum_{i=0}^d m_{uv}(\theta_i)\theta_i^{\ell},
$$
(see Fiol, Garriga and Yebra \cite{fgy99}, Dalf\'o, Fiol and Garriga \cite{dfg09} and Dalf\'o, Van Dam, Fiol, Garriga and Gorissen \cite{ddfgg10}).
In particular, when $u=v$,  $m_u(\theta_i)=m_{uu}(\theta_i)$ are the so-called  {\it local multiplicities} of vertex $u$,
satisfying $\sum_{i=0}^d m_u(\theta_i) = 1$ for $u\in V$, and $\sum_{u\in
V} m_u(\theta_i) =m_i$ for $i=0,1,\ldots,d$ (see Fiol and Garriga \cite{fg97}).

Let $\dist(u,v)$ denote the distance between vertices $u,v$. From any of the above expressions for $\E_i$, we infer that, when $\G$ is bipartite,  the crossed local multiplicities satisfy:
\begin{itemize}
\item
$m_{uv}(\theta_i)=m_{uv}(\theta_{d-i})$,\ \ \ \quad  $i=0,1,\ldots,d$,\quad if $\dist(u,v)$ is even.
\item
$m_{uv}(\theta_i)=-m_{uv}(\theta_{d-i})$, \quad $i=0,1,\ldots,d$,\quad if $\dist(u,v)$ is odd.
\end{itemize}
In particular, as $\dist(u,u)=0$, the local multiplicities bear the same symmetry as the standard multiplicities: $m_u(\theta_i)=m_u(\theta_{d-i})$ for any vertex $u\in V$ and eigenvalue $\theta_i$, $i=0,1,\ldots,d$.

From the above results, notice that, when $\G$ is regular and bipartite, we have $\E_0=\frac 1n \J$ (as mentioned before) and
\begin{equation}
\label{Ed-bip}
 \E_d  =  \frac 1n \left(
\begin{array}{rr}
\J & -\J \\
-\J & \J
\end{array}
 \right).
\end{equation}

As proved by Haemers \cite{h95}, the following result is a direct consequence of the {\em Interlacing Theorem}.

\begin{theo}[\cite{h95}]
\label{theo:tesiHaemers1.2.3}
Let $\A$ be a symmetric matrix  partitioned into $m^2$ blocks
$\A_{ij}$, \linebreak $i,j=1,2,\ldots,m$, with $m< n$, such that $\A_{ii}$ is a square matrix for any $i=1,2,\ldots,m$. Let
$\B=(b_{ij})$ be  the $m\times m$ matrix with $b_{ij}$ being the average row sum of $\A_{ij}$, for
$i,j=1,2,\ldots,m$. Let  $\lambda_{1}\geq
\lambda_{2}\geq \cdots \geq \lambda_{n}$ and $\mu_{1}\geq \mu_{2}\geq \cdots \geq \mu_{m}$ be the eigenvalues of $\A$ and $\B$, respectively. Then,
\begin{itemize}
\item[$(a)$]
The eigenvalues of $\B$ interlace the eigenvalues of $\A$. That is,
$$
\lambda_{n-m+i}\le \mu_i\le \lambda_i, \quad i=1,2,\ldots,m.
$$
\item[$(b)$]
If the interlacing is tight, that is, for some $0\le k\le m$, $\mu_i=\lambda_i$ for $i=1,2,\ldots,k$ and $\mu_i=\lambda_{n-m+i}$ for $i=k+1,k+2,\ldots,m$, then $\A_{ij}$ has constant row
and column sums for $i,j=1,2,\ldots,m$.
\end{itemize}
\end{theo}

\section{Regular graphs}

We begin our study with a simple result characterizing biregularity in a bipartite graph.

\subsection{Spectrum and regularity}
For a given graph $\G=(V,E)$, let $\delta_u$ denote the degree of vertex $u\in V$. Then,
the {\em average degree} of $\G$  is defined by
\begin{equation}\label{average-degree}
\overline{\delta}=\frac{1}{n}\displaystyle \sum_{u\in
V}\delta_{u}=\frac{1}{n}\tr \A^{2}=\frac{1}{n}\displaystyle
\sum_{i=1}^{n}\lambda_{i}^{2}.
\end{equation}
Thus, if in Theorem \ref{theo:tesiHaemers1.2.3} we consider  the trivial partition with $m=1$, the {\em average quotient matrix} $\B$ (whose entries are the average row sums of the blocks of $\A$) has eigenvalue $\mu_{1}=\overline{\delta}$ and, hence,
\begin{equation}
\label{ineq}
\overline{\delta}\leq \lambda_{1},
\end{equation}
with equality if and only if $\G$ is $\overline{\delta}$-regular. (Combining this with (\ref{average-degree}), we see that regularity is a property that can be deduced from the spectrum.)

Let us show that there is an analogous result for bipartite graphs. A
bipartite graph $\G=(V_{1}\cup V_{2},E)$ is called
$(\delta_{1},\delta_{2})$-\emph{biregular}
when the $n_{1}$ vertices of $V_{1}$ have degree
$\delta_{1}$, and the $n_{2}$ vertices of $V_{2}$ have
degree $\delta_{2}$. Thus, by counting in two different
ways the number of edges $m=|E|$, we have that
$n_{1}\delta_{1}=n_{2}\delta_{2}$. Also, it is well known that, for such a graph, $\theta_0=\sqrt{\delta_1\delta_2}$ (see, for instance, Godsil and Royle \cite[pp. 172--173]{gr01}). For a general bipartite graph with independent sets $V_1,V_2$, define $\overline{\delta}_{1}$ and $\overline{\delta}_{2}$ as the average degree of the vertices in $V_{1}$ and $V_{2}$, that is,
$\overline{\delta}_{1}=\frac{1}{n_{1}}\sum_{u\in
V_{1}}\delta_{u}$  and $\overline{\delta}_{2}=\frac{1}{n_{2}}\sum_{u\in
V_{2}}\delta_{u}$.

\begin{propo}
\label{propo:Aida1}
Let $\G=(V_{1}\cup V_{2},E)$ be a bipartite graph with $n=n_{1}+n_{2}$ vertices, average degrees $\overline{\delta}_{1}$ and $\overline{\delta}_{2}$ and maximum eigenvalue $\lambda_1$. Then,
\begin{equation}\label{ineq-bipartite}
\sqrt{\overline{\delta}_{1}\overline{\delta}_{2}}\le \lambda_{1}
\end{equation}
and equality holds if and only if $\G$ is $(\overline{\delta}_{1},\overline{\delta}_{2})$-biregular.
\end{propo}

\begin{proof}
Let
$$
\A=\left( \begin{array}{cc}
\vec0 & \A_{1,2} \\
\A_{2,1} & \vec0
 \end{array} \right)
 $$
be the adjacency matrix of $\G$ and  consider the natural average quotient matrix with $m=2$
$$
\B=\left( \begin{array}{cc}
0 & \overline{\delta}_{1} \\
\overline{\delta}_{2} & 0
 \end{array} \right).
 $$
Then, since its eigenvalues are $\pm \sqrt{\overline{\delta}_{1}\overline{\delta}_{2}}$, Theorem \ref{theo:tesiHaemers1.2.3}$(a)$ gives
$$
\mu_{1}=\sqrt{\overline{\delta}_{1}\overline{\delta}_{2}}\le \lambda_1.
$$
Moreover, in case of equality, $\mu_2=-\mu_1=-\lambda_1=\lambda_n$, so that the interlacing is tight
and Theorem
\ref{theo:tesiHaemers1.2.3}$(b)$ implies the biregularity of $\G$.
\end{proof}

In fact, we can derive (\ref{ineq}) and (\ref{ineq-bipartite}) and other similar results from the well-known result from linear algebra known as the {\em Rayleigh's principle} or {\em Rayleigh's inequalities} (see, for instance, Godsil and Royle \cite[p. 202]{gr01}).
This result states that, if
$\u_{1},\u_{2},\ldots,\u_{n}$ are the eigenvectors corresponding
to $\lambda_{1}\geq \lambda_{2}\geq \cdots \geq \lambda_{n}$
respectively, and for some $1\leq i\leq j\leq n$ we have $\u \in
\span \{\u_{i},\u_{i+1},\ldots,\u_{j}\}$, then
$$
\label{Rayleigh}
\lambda_{j}\leq \frac{\langle \u,\A\u\rangle}{||\u||^{2}}\leq
\lambda_{i}.
$$
Moreover, equality on the left (respectively, right) implies that $\u$ is a $\lambda_{j}$-eigenvector (respectively, $\lambda_{i}$-eigenvector) of $\A$.
Then, the clue is to make the ``right choice'' of $\u$. For instance, let us
consider the derivation of the two previous results:

\begin{enumerate}
\item
If $\u=\j$, we have $\langle \j,\A\j \rangle=
\displaystyle \sum_{u\in V}\delta_{u}$, $||\j||^{2}=n$, and we get
(\ref{ineq}).

\item
Assume that $\G$ is bipartite with stable sets $V_{1}$, $V_{2}$,
numbers of vertices $n_{1}=|V_{1}|$, $n_{2}=|V_{2}|$, and average
degrees $\overline{\delta}_{1}$, $\overline{\delta}_{2}$.
Notice that
$n_{1}\overline{\delta}_{1}+n_{2}\overline{\delta}_{2}=2m$. Then,
if
$\u=(\sqrt{\overline{\delta}_{1}}\j|\sqrt{\overline{\delta}_{2}}\j)$
with $||\u||^{2}=2m$, we get

$$
\frac{\langle \u,\A\u \rangle}{||\u||^{2}}
=\sqrt{\overline{\delta}_{1}\overline{\delta}_{2}}\leq
\lambda_{1}
$$

and, from the above comments, equality is attained when $\G$
is biregular.

\end{enumerate}

\subsection{Polynomials and regularity}

The {\em predistance polynomials} $p_0,p_1,\ldots,p_d$, with degree $\deg p_i=i$, $i=0,1,\ldots,d$, associated to a given graph $\G$ with spectrum $\sp \G$ as in (\ref{spG}), are a sequence of orthogonal polynomials with respect to the scalar product
$$
\langle f,g\rangle_\G = \frac 1n \tr(f(\A)g(\A))= \frac 1n \sum_{i=0}^d m_i f(\theta_i)g(\theta_i),
$$
normalized in such a way that $\|p_i\|_{\G}^2=p_i(\theta_0)$ (this makes sense since, as it is well-known, we always have $p_i(\theta_0)>0$;
for a simple explanation see, for instance, Van Dam \cite{vd08}).
As every sequence of orthogonal polynomials, the predistance polynomials satisfy a three-term recurrence of the form
\begin{equation}\label{recur-pol}
xp_i=\beta_{i-1}p_{i-1}+\alpha_i p_i+\gamma_{i+1} p_{i+1},\qquad i=0,1,\ldots,d,
\end{equation}
where $\beta_{-1}=\gamma_{d+1}=0$ and the other constants $\beta_{i},\gamma_{i}$ are nonzero, initialized with $p_0=1$ and $p_1=x$.
Moreover, if $\G$ is bipartite,  the symmetry of such a scalar product yields that $\alpha_i=0$ for every $i=1,2,\ldots,d$, and $p_i$ is even (respectively, odd) for even (respectively, odd) degree  $i$.
In terms of the predistance polynomials, the {\em Hoffman polynomial} is  $H=p_0+p_1+\cdots+p_d$, and satisfies $H(\theta_0)=n$ (the order of the graph) and $H(\theta_i)=0$ for $i=1,2,\ldots,d$ (for more details about the predistance polynomials, see
C\'amara, F\`abrega, Fiol and Garriga \cite{cffg09}).
In \cite{hof63}, Hoffman proved that a (connected) graph $\G$ is regular if and only if $H(\A)=\J$.
(In fact, $H$ is the unique polynomial of degree at most $d$ satisfying this property.)
Furthermore, when $\G$ is regular and bipartite, the even and odd parts of $H$,  $H_0=\sum_{i\ {\rm even}} p_i$ and $H_1=\sum_{i\ {\rm odd}} p_i$, satisfy, by (\ref{p(A)}):
\begin{equation}
\label{H0(A)-H1(A)}
H_0(\A)  =  \left(
\begin{array}{cc}
\J & \O \\
\O & \J
\end{array}
 \right)\quad \mbox{ and }\quad
H_1(\A)  =
 \left(
\begin{array}{cc}
\O & \J \\
\J & \O
\end{array}
 \right).
\end{equation}

As far as we know, the following theorem is new and can be seen as the biregular counterpart of Hoffman's result.
\begin{theo}
 A bipartite graph $\G=(V_1\cup V_2, E)$  with $n=|V_1|+|V_2|=n_1+n_2$ vertices is $(\delta_1,\delta_2)$-biregular if and only if the odd part of its Hoffman polynomial satisfies
\begin{equation}\label{H1(A)}
H_1(\A)=\alpha\left(
\begin{array}{cc}
\O & \J \\
\J & \O
\end{array}
 \right)
\end{equation}
with 
$\displaystyle \alpha=\frac{n_1+n_2}{2\sqrt{n_1 n_2}}=\frac{\delta_1+\delta_2}{2\sqrt{\delta_1 \delta_2}}$.

\end{theo}
\label{Hof-bip(1)}
\begin{proof}
Assume first that $\G$ is biregular with degrees, say, $\delta_1$ and $\delta_2$. Then, \linebreak $\theta_0=-\theta_d=\sqrt{\delta_1\delta_2}$ with respective (column) eigenvectors $\u=(\sqrt{\delta_1}\j|\sqrt{\delta_2}\j)^{\top}$ and $\v=(\sqrt{\delta_1}\j|-\sqrt{\delta_2}\j)^{\top}$, with the $\j$'s being all-$1$ (row) vectors with the appropriate lengths. Therefore, the respective idempotents are
\begin{eqnarray*}
\E_0 & = & \frac{1}{\|\u\|^2}\u\u^{\top}=\frac{1}{n_1\delta_1+n_2\delta_2} \left(
\begin{array}{cc}
\delta_1\J & \sqrt{\delta_1\delta_2}\J \\
\sqrt{\delta_1\delta_2}\J & \delta_2\J
\end{array}
 \right), \\
 \E_d & = & \frac{1}{\|\v\|^2}\v\v^{\top}=\frac{1}{n_1\delta_1+n_2\delta_2} \left(
\begin{array}{cc}
\delta_1\J & -\sqrt{\delta_1\delta_2}\J \\
-\sqrt{\delta_1\delta_2}\J & \delta_2\J
\end{array}
 \right).
\end{eqnarray*}
As $H_1(x)=\frac 12 [H(x)-H(-x)]$ with $H(\theta_0)=n$ and $H(\theta_i)=0$ for any $i\neq 0$, we have that
$H_1(\theta_0)=-H_1(\theta_d)=n/2$ and $H_1(\theta_i)=0$ for $i\neq 0,d$.
Hence, using the properties and the above expressions of the idempotents,
\begin{eqnarray*}
H_1(\A) & = & \sum_{i=0}^d H_1(\theta_i)\E_i=H_1(\theta_0)\E_0+H_1(\theta_d)\E_d \nonumber\\
 & = & \frac{n}{2}(\E_0-\E_d)=
\frac{n\sqrt{\delta_1\delta_2}}{n_1\delta_1+n_2\delta_2}\left(
\begin{array}{cc}
\O & \J \\
\J & \O
\end{array}
 \right),
\end{eqnarray*}
and
the result follows since $n_1\delta_1=n_2\delta_2$.

Conversely, if (\ref{H1(A)}) holds,  and $\textstyle\A=\left(
\begin{array}{cc}
\O & \B \\
\B^\top & \O
\end{array}
 \right)$,
 the equality $\A H_1(\A)=H_1(\A)\A$ yields
$$
\left(
\begin{array}{cc}
\B\J & \O \\
\O & \B^{\top}\J
\end{array}\right)
 =
 \left(
\begin{array}{cc}
\J\B^{\top} & \O \\
\O & \J\B
\end{array}
 \right).
$$
Thus, $(\B\J)_{uv}=(\J\B^{\top})_{uv}$ implies that $\delta_u=\delta_v$ for any two vertices $u,v\in V_1$, whereas $(\B^{\top}\J)_{wz}=(\J\B)_{wz}$ means that $\delta_w=\delta_z$ for any two vertices $w,z\in V_2$. Thus, $\G$ is biregular and the proof is complete.
\end{proof}

Notice that the constant $\alpha$ is the ratio between the arithmetic and geometric means of the numbers $n_1,n_2$. Hence, (\ref{H1(A)}) holds with $\alpha=1$ if and only if $n_1=n_2$ or, equivalently, when $\G$ is regular.

In fact, the above result can be reformulated in the following way:

\begin{theo}
\label{Hof-bip(2)}
A bipartite graph $\G$, with $n=n_1+n_2$ vertices and distinct eigenvalues $\theta_0>\theta_1>\cdots>\theta_d$,
is connected and biregular if and only if there exists a
polynomial $P$ satisfying 
\begin{equation}\label{P(A)}
P(\A)=\left(
\begin{array}{cc}
\O & \J \\
\J & \O
\end{array}
 \right).
\end{equation}
Moreover, $P=\frac{2\sqrt{n_1 n_2}}{n_1+n_2}H_1$, where $H_1$ is the odd part of the Hoffmann  polynomial of $\G$, so that   $P(\theta_0)=-P(\theta_d)=\sqrt{n_1 n_2}$, and $P(\theta_i)=0$ for $i\neq 0,d$.
\end{theo}

\begin{proof}
We only need to prove sufficiency, since necessity has already been  proved in Theorem \ref{Hof-bip(1)}. If (\ref{P(A)}) holds, we have that for any two vertices $u,v$ in different partite sets, $(P(\A))_{uv}=1\neq 0$ implies that there is some $u$-$v$ path. Also, if $u,v$ are vertices in the same partite set, for any vertex $w$ in the other partite set there exist $u$-$w$ and $v$-$w$ paths, thus assuring again the existence of some $u$-$v$ path. Consequently, $\G$ is connected and, from the same reasoning as above, biregular with degrees, say, $\delta_1$, $\delta_2$. Now, let us prove that $P$ is as claimed. Since $\u=(\sqrt{\delta_{1}}\j|\sqrt{\delta_{2}}\j)^{\top}$ is an eigenvector of $\A$ for the eigenvalue $\theta_0=\sqrt{\delta_1\delta_2}$, we have
$$
P(\A)\u=P(\theta_0) \left(
\begin{array}{c}
\sqrt{\delta_{1}}\j^{\top} \\
\sqrt{\delta_{2}}\j^{\top}
\end{array}
 \right)=
 \left(
\begin{array}{cc}
\O & \J \\
\J & \O
\end{array}
 \right)
 \left(
\begin{array}{c}
\sqrt{\delta_{1}}\j^{\top} \\
\sqrt{\delta_{2}}\j^{\top}
\end{array}
 \right)=
 \left(
\begin{array}{c}
n_2\sqrt{\delta_{2}}\j^{\top} \\
n_1\sqrt{\delta_{1}}\j^{\top}
\end{array}
 \right),
$$
whence $P(\theta_0)=\frac{n_2\sqrt{\delta_{2}}}{\sqrt{\delta_{1}}}=\sqrt{n_1n_2}$. Similarly, by using that $\v=(\sqrt{\delta_{1}}\j|-\sqrt{\delta_{2}}\j)^{\top}$ is an eigenvector of $\A$ for the eigenvalue $\theta_d=-\sqrt{\delta_1\delta_2}$, we get $P(\theta_d)=-\sqrt{n_1n_2}$.
Finally, assume that $\w=(\w_1|\w_2)^{\top}$ is an eigenvector of $\A$ with  eigenvalue $\theta_i$, $i\neq 0,d$. Then, from the orthogonality conditions $\langle \w,\u\rangle=0$ and $\langle \w,\v\rangle=0$ (where $\u$ and $\v$ are as in the proof of Theorem \ref{Hof-bip(1)}) we have, respectively, that
$\sqrt{\delta_1}\langle \w_1,\j\rangle+\sqrt{\delta_2}\langle \w_2,\j\rangle=0$ and $\sqrt{\delta_1}\langle \w_1,\j\rangle-\sqrt{\delta_2}\langle \w_2,\j\rangle=0$. Therefore, $\langle \w_1,\j\rangle=\langle \w_2,\j\rangle=0$. Moreover,
$$
P(\A)\w=P(\theta_i)
\left(
\begin{array}{c}
\w_1^{\top} \\
\w_2^{\top}
\end{array}
 \right)=
 \left(
\begin{array}{cc}
\O & \J \\
\J & \O
\end{array}
 \right)
\left(
\begin{array}{c}
\w_1^{\top} \\
\w_2^{\top}
\end{array}
 \right)=
 \left(
\begin{array}{c}
\langle \w_2,\j\rangle\j^{\top} \\
\langle \w_1,\j\rangle\j^{\top}
\end{array}
 \right)=
 \left(
\begin{array}{c}
\vec0 \\
\vec0
\end{array}
 \right),
$$
and, since $\w\neq \vec0$, $P(\theta_i)=0$, as claimed.
\end{proof}

\section{Distance-regular graphs}
\label{drg}
In this section we give some new characterizations of bipartite distance-regular graphs. As commented in the introduction, a general phenomenon is that the conditions for a bipartite graph to be distance-regular can be relaxed in comparison with the analogues for general graphs. To prove and compare our results, we first recall some basic facts from the theory of distance-regular graphs and some of their known characterizations.

\subsection{Some characterizations of distance-regularity}
Let $\G$ be a graph with diameter $D$, adjacency matrix
$\A$, and $d+1$ distinct eigenvalues. Let $\A_i$, $i=0,1,\ldots $, be the {\em distance-$i$ matrix} of $\G$, with entries  $(\A_i)_{uv}=1$ if $\dist(u,v)=i$ and $(\A_i)_{uv}=0$ otherwise (if $i>D$, we have  $\A_i=\matrix0$). Then,
$$
{\cal A}=
\mathbb{R}_{d}[\A]=\span \{\textbf{\emph{I}},
\textbf{\emph{A}}, \textbf{\emph{A}}^2, \ldots,
\textbf{\emph{A}}^{d}\}
$$
is an algebra, with the ordinary product of matrices and orthogonal basis  $\{\E_0,\E_1,\ldots,
\E_d\}$ and $\{p_0(\A),p_1(\A),\ldots, p_d(\A)\}$, called
the {\it adjacency  algebra}; whereas
$$
{\cal D}= \span
\{\textbf{\emph{I}},\textbf{\emph{A}},\textbf{\emph{A}}_2,\ldots,\textbf{\emph{A}}_D\}
$$
forms an algebra with the entrywise or Hadamard product ``$\circ$" of matrices, defined
by
$(\textbf{\emph{X}}\circ\textbf{\emph{Y}})_{uv}=(\textbf{\emph{X}})_{uv}(\textbf{\emph{Y}})_{uv}$.
We call ${\mathcal D}$ the {\em distance $\circ$-algebra}.
Note that, when $\G$ is regular,  $\I,\A,\J\in {\cal A}\cap {\cal D}$ since
$\J=H(\A)=\sum_{i=0}^D \A_i$. Thus, $\textrm{dim}\, ({\cal
A}\cap {\cal D}) \geq 3$, if $\G$ is not a complete graph (in
this exceptional case, $\textbf{\emph{J}}=\textbf{\emph{I}}+\textbf{\emph{A}}$).
In this algebraic context, an important result is that $\G$ is distance-regular if and only if ${\mathcal A}={\mathcal D}$,
 which is therefore equivalent to $\textrm{dim}\, ({\cal A}\cap {\cal D}) =
d+1$ (and hence $d=D$); see, for instance, Biggs \cite{biggs} or Brouwer, Cohen and Neumaier \cite{bcn}. This leads to some characterizations of distance-regularity in terms of some constants $p_{ji}$, $q_{ij}$ and $a_i^{(\ell)}$ (see also Rowlinson \cite{r97}): A graph $\G$ is distance-regular if and only if any of the following conditions holds.
\begin{itemize}
 \item[$(a)$]
 $\A_i\E_j  =  p_{ji}\E_j$,\quad $i,j=0,1,\ldots, d(=D)$.
 \item[$(b)$]
 $\E_j\circ \A_i = q_{ij}\A_i$,\quad $i,j=0,1,\ldots, d$.
\item[$(c)$]
$\A^{\ell}\circ \A_i =a_i^{(\ell)}\A_i$,\quad $i,\ell=0,1,\ldots, d$.
\end{itemize}

When the conditions $(a)$ hold, the predistance polynomials become the {\em distance polynomials} satisfying
$p_i(\A)=\A_i$, $i=0,1,2,\ldots, d$
(see, for example, Bannai and Ito \cite{bi84}).
Moreover, for general graphs with not necessarily $D=d$,
the conditions $(a)$ are a characterization of the
so-called {\em distance-polynomial graphs}, introduced by  Weichsel
\cite{w82} (see also Beezer \cite{beezer} and Dalf\'o, van Dam, Fiol, Garriga and Gorissen \cite{ddfgg10}). This is equivalent
to ${\cal D}\subset {\cal A}$ (but not necessarily ${\cal D}=
{\cal A}$); that is, every distance matrix
$\textbf{\emph{A}}_i$ is a polynomial (not necessarily $p_i$) in $\textbf{\emph{A}}$.
In contrast with this,  the conditions $(b)$ and $(c)$ are
equivalent to ${\cal A}\subset {\cal D}$ and hence to ${\cal
A}={\cal D}$ (which implies $d=D$) as $\textrm{dim}\, {\cal
A}\ge D+1=\textrm{dim}\, {\cal D}$.

As summarized in the following theorem, the above conditions for  distance-regularity can be relaxed if
we assume some extra natural hypothesis (such as regularity) thus giving more `economical' characterizations.

\begin{theo}[\cite{fgy1b,f01,ddfgg10}]
\label{theo-charac-drg}
A regular graph $\G$ with $d+1$ distinct eigenvalues, diameter $D=d$, and idempotents $\E_1,\E_d$ is distance-regular if and only if any of the following conditions holds:
\begin{itemize}
\item[$(a)$]
$\A_d\in {\cal A}$,\quad or\quad $\A_d = p_d(\A)$,\quad or\quad $\A_i = p_i(\A)$ for $i=d-2,d-1$.
\item[$(b)$] $\E_j\in {\cal D}$  for  $j=1,d$.
\item[$(c)$] $\A^{\ell}\circ \A_i \in {\cal D}$ for $\ell=i,i+1$ and $i\le d-1$.
\end{itemize}
\end{theo}

\subsection{The case of bipartite graphs}
In the following result we prove that for bipartite graphs the above
conditions can be further relaxed giving new characterizations for bipartite
distance-regular graphs.

\begin{theo}
\label{charac-drbg}
A regular bipartite graph $\G$ with diameter $D=d$, idempotent $\E_1$, and predistance polynomials $p_0,\ldots, p_d$ is dis\-tance-regular if and only if any of the following conditions holds:
\begin{itemize}
\item[$(a1)$]
$\A_{i}=p_{i}(\A)$ for $i=d-3,d-2$, $d\geq3$.
\item[$(a2)$]
$\A_{i}=p_{i}(\A)$ for $i=d-4,d-2$, $d\geq4$.
\item[$(b)$]
$\E_1\in {\cal D}$.
\item[$(c)$]
$a_{uv}^{(\ell)}=a_i^{(\ell)}$  for  $\ell=i\le D-2$.
\end{itemize}
\end{theo}

\begin{proof}
We only prove sufficiency because necessity is straightforward.

$(a1)$ From Theorem \ref{theo-charac-drg}$(a)$, it suffices to prove that $p_{d-1}(\A)=\A_{d-1}$. Let $u,v$ be two vertices at distance $\dist(u,v)=i$.
First notice that, if $i$ and $d-1$ have distinct parity, $(p_{d-1}(\A))_{uv}=0$. In particular, this is the case when $i=d$.
Thus, if $i=d-1$, we have $(p_{d-1}(\A))_{uv}=(H(\A))_{uv}=1$.
Moreover, if $i=d-3$ we have $(p_{d-3}(\A))_{uv}+(p_{d-1}(\A))_{uv}=1$, where, from the hypothesis,  $(p_{d-3}(\A))_{uv}=1$. Hence, $(p_{d-1}(\A))_{uv}=0$.
Thus, the only case left is when $i\le d-5$ and $i$ has the same parity as $d-1$.
In this case the two-term recurrence $xp_j=\beta_{j-1}p_{j-1}+\gamma_{j+1}p_{j+1}$, for $j=0,\ldots,d$, satisfied by the predistance polynomials (of a bipartite graph) yields for $j=d-2$
$$
(\A p_{d-2}(\A))_{uv}=\beta_{d-3}(\A_{d-3})_{uv}+\gamma_{d-1} (p_{d-1}(\A))_{uv}=\gamma_{d-1} (p_{d-1}(\A))_{uv},
$$
with first term, using again the hypothesis,
$$
(\A p_{d-2}(\A))_{uv}=\sum_{w\in V}a_{uw}(\A_{d-2})_{wv}=\sum_{w\in \G(u)}(\A_{d-2})_{wv}=0,
$$
since $\dist(v,w)\le d-4$. Consequently, as $\gamma_{d-1}\neq 0$, $(p_{d-1}(\A))_{uv}=0$ and $p_{d-1}(\A)=\A_{d-1}$, as claimed.

$(a2)$ Now, from Theorem \ref{theo-charac-drg}$(a)$, it suffices to prove that $p_{d}(\A)=\A_{d}$. Then the proof is similar to the previous one. Indeed, with the same notation as before, if $i=d$, $(p_d(\A))_{uv}=(H(\A))_{uv}=1$, and
if $i$ and $d$ have distinct parity, $(p_{d}(\A))_{uv}=0$. The cases $i=d-2,d-4$ are proved again by using that $p_0+\cdots+p_d=H$. Then, the only case left is when $i\le d-6$ and $i$ and $d$ have the same parity. Now, by applying two times the above two-term recurrence of the predistance polynomials, we have
$x^2p_{d-2}=B_{d-4}p_{d-4}+A_{d-2}p_{d-2}+C_{d}p_d$, for some constants $B_{d-4}$, $A_{d-2}$ and $C_{d}$. Hence, from the hypothesis,
$$
(\A^2 \A_{d-2}){uv}=B_{d-4}(\A_{d-4})_{uv}+A_{d-2}((\A)_{d-2})_{uv}+C_{d} (p_{d}(\A))_{uv}=C_{d} (p_{d}(\A))_{uv},
$$
where
$$
(\A^2 \A_{d-2}){uv}=\sum_{w\in V}(\A^2)_{uw}(\A_{d-2})_{wv}=\sum_{\dist(w,u)\le 2}(\A_{d-2})_{wv}=0,
$$
since $\dist(w,v)\le d-4$. Then, as $C_d=\gamma_{d-1}\gamma_d\neq 0$,
$(p_{d}(\A))_{uv}=0$ and $p_{d}(\A)=\A_{d}$, as required.

Case $(b)$ is a consequence of Theorem \ref{theo-charac-drg}$(b)$ since, under the hypotheses, (\ref{Ed-bip}) yields
$$
\textstyle
\E_d=\frac{1}{n}\left(\sum_{\stackrel{\mbox{\scriptsize $i=0$}}{i\ {\rm even}}}^{d}\A_i-\sum_{\stackrel{\mbox{\scriptsize $i=0$}}{i\ {\rm odd}}}^{d}\A_i\right) \in {\cal D}.
$$
Finally, $(c)$ follows from
the fact that, since $\G$ is bipartite, there are no walks of length $\ell=i+1$ between vertices at distance $i$ and, thus, $a_i^{(i+1)}=0$. Moreover, as $\G$ is $\delta$-regular and $D=d$, we have $a_{d-1}^{(d-1)} = \frac{1}{\delta}a_{d}^{(d)}=\frac{\pi_0}{n\delta}$
because the number of $d$-walks between any two vertices $u,v$ at distance $d$, is a constant:
$
a_{uv}^{(d)}  =  (\A^d)_{uv}=\frac{\pi_0}{n}[H(\A)]_{uv}=\frac{\pi_0}{n}=a_{d}^{(d)}.
$
\end{proof}

In \cite{vd95} Van Dam proved that a connected regular graph with four distinct eigenvalues is {\em walk-regular} (that is, for every $\ell\ge 0$, the number of closed $\ell$-walks rooted at a given vertex is a constant through the graph, see Godsil and McKay \cite{gmk80}). Also, Dalf\'o, van Dam, Fiol, Garriga and Gorissen \cite{ddfgg10} observed that the same is true for a every connected regular bipartite graph with five distinct eigenvalues. (For properties about bipartite biregular graphs with five eigenvalues, see Van Dam and Spence \cite{vds05}.) As a consequence of the above theorem, we also have the following:

\begin{coro}
\begin{itemize}
\item[$(a)$]
Every bipartite $\delta$-regular graph $\G$ with $d+1=4$ distinct eigenvalues and diameter $D=3$ is distance-regular.
\item[$(b)$]
Every bipartite $\delta$-regular graph $\G$ with $d+1=5$ distinct eigenvalues, diameter $D=4$, and with every pair of nonadjacent vertices having $c$ common neighbors is distance-regular.
\end{itemize}
\end{coro}
\begin{proof}
Just notice that, in case $(a)$, $p_{d-2}=p_1=x$ is the distance-$1$ polynomial, whereas in $(b)$, under the hypothesis, $p_{d-2}=p_2=\frac{1}{c}(x^2-\delta)$ is the distance-$2$ polynomial. Hence, in both cases,
Theorem \ref{charac-drbg}$(a)$ applies.
\end{proof}

\subsection{The spectral excess theorem for bipartite graphs}
Let $\G$ be a graph with $d+1$ eigenvalues and distance-$d$ matrix $\A_d$. Then the {\em excess} of a vertex $u$ is the number of vertices at distance $d$ from $u$, that is $\exc(u)=|\G_d(u)|$, and the {\em average excess} is  $\overline{\delta}_d=\frac{1}{n}\sum_{u\in V}\exc(u)$. Moreover, the {\em spectral excess} is the value at $\lambda_0$ of the highest degree predistance polynomial and it can be computed by using $\sp \G$ through the formula
$$
p_d(\lambda_0)=n\left(\sum_{i=0}^d\frac{\pi_0^2}{\pi_i^2} \right)^{-1},
$$
where $\pi_i=\prod_{j\neq i}|\lambda_i-\lambda_j|$ for $i=0,1,\ldots,d$.
By using these parameters, the spectral excess theorem, due to Fiol and Garriga \cite{fg97}, provides a quasi-spectral characterization of distance-regularity for regular connected graphs and reads as follows (see Van Dam \cite{vd08} or Fiol, Gago and Garriga \cite{fgg10} for short proofs):

\begin{theo}[\cite{fg97,vd08,fgg10}]
A connected regular graph with $d+1$ distinct eigenvalues is distance-regular if and only if its average excess equals its spectral excess:
\begin{equation}\label{SPET}
\overline{\delta}_d=p_d(\lambda_0).
\end{equation}
\end{theo}

The proof of this theorem is based on the characterization of Theorem \ref{theo-charac-drg}$(a)$ and, hence, it is natural to devise an analogous result for bipartite graphs, which is given in terms of the following parameters.
Given a graph $\G$, let $\omega_i$ be the leading coefficient of its predistance polynomial $p_i$, and consider the average numbers  of vertices at distance $i$ and shortest $i$-paths; that is,
$$
\overline{\delta}_{i}=\frac{1}{n}\sum_{u\in V}|\G_{i}(u)|\quad \mbox{and}\quad \overline{a}_{i}^{(i)}=\frac{1}{n\overline{\delta}_{i}}\sum_{\dist(u,v)=i}a_{uv}^{(i)},
$$
respectively. Then, we have the following result:

\begin{theo}
A regular bipartite graph $\G$ with $d+1$ distinct eigenvalues and diameter $D=d$ is distance-regular if and only if
\begin{equation}\label{SPET-bip}
\overline{a}_{i}^{(i)}=1/\omega_{i}\quad
\mbox{and}\quad\overline{\delta}_{i}=p_{i}(\lambda_0).
\end{equation}
for either $i=d-3,d-2$ or $i=d-4,d-2$. 
\end{theo}
\begin{proof}
The necessity of both conditions follows easily by considering that, if $\G$
is distance-regular, then $p_{i}(\A)=\A_{i}$, for every $i=0,\ldots,d$; whereas
the sufficiency follows from \cite[Prop. 2.4$(a3)$]{ddfg12}, assuring that $p_{i}(\A)=\A_{i}$ for either $i=d-3,d-2$ or $i=d-4,d-2$, and Theorem \ref{charac-drbg}$(a)$.
\end{proof}

In fact, as shown in \cite{ddfg12}, the two conditions in (\ref{SPET-bip}) can be replaced by only one, namely,
$$
\overline{\delta}_{i}=\frac{p_{i}(\lambda_0)}{\omega_{i}^2[\overline{a}_{i}^{(i)}]^2},
$$
for either $i=d-3,d-2$ or $i=d-4,d-2$.
to be compared with (\ref{SPET}).
\vskip 1cm

\noindent {\bf Acknowledgments.}
Research supported by the Ministerio de Ciencia e
Innovaci\'on (Spain) and the European Regional Development Fund under
project MTM2011-28800-C02-01, and by the Catalan Research Council
under project 2009SGR1387.


\begin{thebibliography}{99}

\bibitem{bi84}
E. Bannai and T. Ito, \emph{Algebraic Combinatorics I: Association Schemes}, Benjamin/Cummings Menlo Park, CA, 1984.

\bibitem{biggs}
N. Biggs, \emph{Algebraic Graph Theory}, Cambridge University Press,
Cambridge, 1974, second edition, 1993.

\bibitem{beezer}
R.A. Beezer, Distance polynomial graphs, in
{\em Proc. Sixth Caribbean Conf. on
Combin. and Computing}, Trinidad (1991) 51--73.

\bibitem{bcn}
A.E. Brouwer, A.M. Cohen, and A. Neumaier, \emph{Distance-Regular Graphs},
Springer-Verlag, Berlin-New York, 1989.

\bibitem{bh12}
A.E. Brouwer and W.H. Haemers, \emph{Spectra of Graphs},
Springer-Verlag, Berlin-New York, 2012; available online at
\url{http://homepages.cwi.nl/~aeb/math/ipm/}.

\bibitem{cffg09}
M. C\'amara, J. F\`abrega, M.A. Fiol, and E. Garriga,
Some families of orthogonal polynomials of a discrete variable and
their applications to graphs and codes, {\em Electron. J. Combin.} {\bf 16(1)} (2009), \#R83.

\bibitem{cds82}
D. M. Cvetkovi\'c, M. Doob and H. Sachs, {\em Spectra of Graphs. Theory and Application},
VEB Deutscher Verlag der Wissenschaften, Berlin, second edition, 1982.

\bibitem{ddfgg10}
C. Dalf\'{o}, E.R. van Dam, M.A. Fiol, E. Garriga, and B.L.
Gorissen, On almost distance-regular graphs, {\em J. Combin.
Theory Ser. A} \textbf{118} (2011), 1094--1113.

\bibitem{ddfg12}
C. Dalf\'{o}, E.R. van Dam, M.A. Fiol, and E. Garriga,
Dual concepts of almost distance-regularity and the spectral excess theorem,
{\em Discrete Math.} {\bf 312} (2012), 2730--2734. 

\bibitem{dfg09}
C. Dalf\'{o}, M.A. Fiol, and E. Garriga, On
$k$-walk-regular graphs, \emph{Electron. J. Combin.} {\bf 16(1)} (2009), \#R47.

\bibitem{vd95}
E.R. van Dam, Regular graphs with four eigenvalues, {\em Linear Algebra Appl.} {\bf 226-228}
(1995), 139--162.

\bibitem{vd08}
E.R. van Dam, The spectral excess theorem for distance-regular
graphs: a global (over)view, {\em Electron. J. Combin.} {\bf 15(1)}
(2008), \#R129.

\bibitem{dkt12}
E.R. van Dam, J.H. Koolen, and H. Tanaka, Distance-regular
graphs, manuscript (2012), available online at
\url{http://lyrawww.uvt.nl/~evandam/files/drg.pdf}.

\bibitem{vds05}
E.R. van Dam and E. Spence,
Combinatorial designs with two singular values II. Partial geometric designs,
{\em Linear Algebra  Appl.} {\bf 396} (2005),  303–-316.

\bibitem{f01}
M.A. Fiol, On pseudo-distance-regularity.
\emph{Linear Algebra Appl.} {\bf 323} (2001), 145--165.

\bibitem{f02}
M.A. Fiol, Algebraic characterizations of distance-regular graphs,
\emph{Discrete Math.} {\bf 246} (2002), 111--129.

\bibitem{fg97}
M.A. Fiol and E. Garriga,
From local adjacency polynomials to locally pseudo-distance-regular graphs,
\emph{J. Combin. Theory Ser. B} {\bf  71} (1997), 162--183.

\bibitem{fgg10}
M.A. Fiol, S. Gago, and E. Garriga, A simple proof of the spectral excess theorem for
distance-regular graphs, {\it Linear Algebra Appl.} {\bf 432} (2010), 2418--2422.

\bibitem{fgy1b}
M.A. Fiol, E. Garriga, and J.L.A. Yebra,
Locally pseudo-distance-regular graphs, {\it J.  Combin. Theory Ser. B}
{\bf 68} (1996), 179--205.

\bibitem{fgy99}
M.A. Fiol, E. Garriga, and J.L.A. Yebra,
Boundary graphs: The limit case of a spectral property,
{\it Discrete Math.}  {\bf 226} (2001), 155--173.

\bibitem{g93}
C.D. Godsil, {\it Algebraic Combinatorics}, Chapman and Hall, New York, 1993.

\bibitem{gmk80}
C.D. Godsil and B.D. McKay, Feasibility conditions for the existence of walk-regular graphs, {\it
Linear Algebra Appl.} {\bf 30} (1980), 51--61.

\bibitem{gr01}
C.D. Godsil and G. Royle, {\it Algebraic Graph Theory}, Graduate Texts in Mathematics {\bf 207}, Springer-Verlag,  New York, 2001.

\bibitem{h95}
W.H. Haemers,
Interlacing eigenvalues and graphs,
\emph{Linear Algebra Appl.} {\bf 226-228} (1995), 593--616.

\bibitem{hof63}
A.J. Hoffman, On the polynomial of a graph,
{\it Amer. Math. Monthly} {\bf 70} (1963), 30--36.

\bibitem{r97}
P. Rowlinson, Linear algebra, in {\it Graph
Connections} (L.W. Beineke and R.J. Wilson, eds.), Oxford
Lecture Ser. Math. Appl., Vol. 5, 86--99, Oxford Univ.
Press, New York, 1997.

\bibitem{w82}
P.M. Weichsel, On distance-regularity in graphs, {\it J.  Combin.
Theory Ser. B} {\bf 32} (1982), 156--161.

\end{thebibliography}
\end{document}